\documentclass{amsart}

\usepackage{amsmath,amssymb,color,url}

\newtheorem{lemma}{Lemma}[section]
\newtheorem{prop}[lemma]{Proposition}
\newtheorem{cor}[lemma]{Corollary}
\newtheorem{thm}[lemma]{Theorem}
\newtheorem{example}[lemma]{Example}
\newtheorem{thm?}[lemma]{Theorem?}

\newtheorem{remark}[lemma]{Remark}

\newcommand{\F}{\mathbb{F}}

\newcommand{\ra}{\ensuremath{\rightarrow}}

\newcommand{\Z}{\mathbb{Z}}
\newcommand{\N}{\mathcal{N}}

\newcommand{\pp}{\mathfrak{p}}

\newcommand{\Q}{\mathbb{Q}}

\newcommand{\rank}{\operatorname{rank}}

\newcommand{\mm}{\mathfrak{m}}

\renewcommand{\N}{\mathbb{N}}

\newcommand{\Spec}{\operatorname{Spec}}
\newcommand{\MaxSpec}{\operatorname{MaxSpec}}
\newcommand{\qq}{\mathfrak{q}}

\newcommand{\rk}{\operatorname{rk}}
\newcommand{\rr}{\mathfrak{r}}

\begin{document}
\title{A Note on Rings of Finite Rank}

\author{Pete L. Clark}

\begin{abstract}
The rank $\rk(R)$ of a ring $R$ is the supremum of minimal cardinalities of generating sets of $I$ as $I$ ranges over ideals of $R$.  
 Matson showed that every $n \in \Z^+$ occurs as the rank of some ring $R$.  Motivated by the result of Cohen and Gilmer that a 
ring of finite rank has Krull dimension $0$ or $1$, we give four different constructions of 
rings of rank $n$ (for all $n \in \Z^+$).  Two constructions 
use one-dimensional domains, and the former of these directly generalizes Matson's construction.   Our third construction uses Artinian rings (dimension zero), and our last construction uses polynomial rings over local Artinian rings (dimension one, irreducible, not a domain).
\end{abstract}

\address{Department of Mathematics \\ Boyd Graduate Studies Research Center \\ University
of Georgia \\ Athens, GA 30602\\ USA}

\maketitle
\noindent

\section{Introduction}

\noindent
For a module $M$ over a ring\footnote{Here all rings are commutative and with multiplicative identity.} $R$, let $\mu(M)$ be the minimal cardinality of a set of generators of $M$ as an $R$-module, and let $\mu_*(M)$ be the supremum of $\mu(N)$ as $N$ ranges over 
all $R$-submodules of $M$.     We say $M$ has \textbf{finite rank} if $\mu_*(M) < \aleph_0$.  This implies that $M$ is a Noetherian $R$-module.  We define the \textbf{rank} $\rk(R)$ as $\mu_*(R)$.   In particular, for $n \in \N$, $\rk(R) = n$ means that every 
ideal of $R$ can be generated by $n$ elements and some ideal of $R$ cannot be generated by fewer than $n$ elements.
\\ \indent
This note is directly motivated by the following result.

\begin{thm}(Matson \cite{Matson08})
\label{MATSON} \\
a) For every $N \in \Z^+$, there is a domain $R_N$ with $\rank(R_N) = N$.  \\
b) One may take $R_N$ to be a subring 
of the ring of integers of $\Q(2^{\frac{1}{N}})$.
\end{thm}

\begin{remark}
\label{REMARK1}
Let $\iota: R \ra S$ be a ring map.  For an $R$-module $M$, let $\iota_*(M)$ be the $S$-module $M \otimes S$.  Then 
$\mu(\iota_*(M)) \leq \mu(M)$.
If every ideal of $S$ is $\iota_*(I)$ for some $I$, we get $\rk(S) \leq \rk(R)$.  This holds when $\iota$ is a quotient or a localization map.
\end{remark}
\noindent
For an $R$-module $M$ and a prime ideal $\pp$ of $R$, we put $M_{\pp} = M \otimes R_{\pp}$ and $\mu_{\pp}(M) = \mu(M_{\pp})$.
Remark \ref{REMARK1} gives $\mu_{\pp}(M) \leq \mu(M)$.  By way of a converse,  we have:

\begin{thm}(Forster-Swan \cite{Forster64} \cite{Swan67}) \\
Let $M$ be a finitely generated module over the Noetherian ring $R$.  Then 
\[ \mu(M) \leq \sup_{\pp \in \Spec R} \left( \mu_{\pp}(M) + \dim R/\pp \right). \]
\end{thm}

\begin{example}
\label{EXAMPLE1}
a) We have $\rk(R) = 0$ iff $R$ is the zero ring.  \\
b) We have $\rk(R) = 1$ iff $R$ is a nonzero principal ring (every ideal is principal). \\
c) If $R$ is a nonprincipal Dedekind domain, then $\rk(R) = 2$: apply Forster-Swan.    Or: let $I$ be an ideal of $R$, let $0 \neq x \in R$, and factor $(x)$ as $\pp_1^{a_1} \cdots \pp_r^{a_r}$.  Then $R/(x) \cong \prod_{i=1}^r R/\pp_i^{a_i} \cong \prod_{i=1}^r R_{\pp_i}/(\pp_i R_{\pp_i})^{a_i}$ is principal, 
so $I = \langle x,y \rangle$ for some $y \in I$.  Asano \cite{Asano51} and Jensen \cite{Jensen63} 
showed: if in a domain $R$, for all nonzero ideals $I$ and nonzero $x \in I$ 
there is $y \in I$ with $I = \langle x,y \rangle$, then $R$ is Dedekind.
\end{example}
\noindent
The aim of this note is to explore the class of rings of finite rank with an eye to constructing further families with all 
possible finite ranks.  We begin in $\S$2 with the case of domains.  We review the pioneering work of Cohen and use it 
to deduce a local-global principle for domains of finite rank: Theorem \ref{FIRSTBIGTHM}.  In $\S$3 we show that given any PID $A$ with fraction field 
$F \supsetneq A$ and any field extension $K/F$ of degree $N \in \Z^{\geq 2}$, there is a nonmaximal $A$-order in $K$ of rank $N$, generalizing Theorem \ref{MATSON}.  We also construct, for any $2 \leq n \leq N$, a $\Z$-order in a degree $N$ 
number field with rank $n$.  In $\S$4 we consider the case of rings which are not domains.  We give a general discussion 
of Artinian rings (which always have finite rank) and show that there are Artinian rings of rank $n$ for any $n \in \Z^+$.   Finally we 
determine when a polynomial ring has finite rank, show that for any local Artinian ring $\rr$, the rank of $\rr[t]$ is bounded 
above by the length of $\rr$, and show that we have equality when $\rr$ is moreover principal, so that e.g. for all $n \in \Z^+$, 
$\Z/2^n \Z[t]$ has rank $n$.

\section{Domains of Finite Rank}
\noindent
I.S. Cohen initiated the study of ranks of domains.  We recall some of his results.

\begin{thm}(Cohen \cite{Cohen50})
\label{FIRSTCOHEN}
If $R$ is a domain of finite rank, then $\dim R \leq 1$.
\end{thm}
\noindent
For $R$ Noetherian and $\pp \in \Spec R$, let $k(\pp)$ be the fraction field of $R/\pp$;  put
\[ z_{\pp}(R) = \dim_{k(\pp)} \pp R_{\pp}/\pp^2 R_{\pp}. \]
Then $z_{\pp}(R) \geq  \dim R_{\pp}$; $\pp$ is \textbf{regular} if equality holds, otherwise \textbf{singular}.  Put
\[z(R) = \sup_{\pp \in \MaxSpec R} z_{\pp}(R). \]
Suppose $R$ is a one-dimensional Noetherian \emph{domain}.  Then by
Krull-Akizuki \cite[Thm. 18.7]{CA} the normalization $\overline{R}$ is Dedekind and $\Spec \overline{R} \ra \Spec R$ is surjective and finite-to-one.  By Example \ref{EXAMPLE1}, if $R = \overline{R}$ then $\rk(R) \leq 2$ with equality iff $R$ is not principal.  Henceforth we 
suppose $R$ is \emph{not} normal.  Let \[\mathfrak{c} = (R:\overline{R}) = 
\{ x \in K \mid x \overline{R} \subset R \} \]be the conductor of $R$: it is the largest ideal of $\overline{R}$ which is also an ideal of $R$.  
If $\pp$ is regular then $\pp \overline{R}$ is also prime 
and $R_{\pp} \stackrel{\sim}{\ra} \overline{R}_{\pp \overline{R}}$.  We say $R$ is \textbf{nearly Dedekind} if $\mathfrak{c} \neq 0$; equivalently, if $\overline{R}$ is a finitely generated $R$-module.  In a nearly Dedekind domain, the singular primes are characterized as: the primes $\pp$ such that $\pp + \mathfrak{c} \subsetneq R$; the radicals of the ideals in a primary decomposition of $\mathfrak{c}$; or the primes of $R$ lying under primes of $\overline{R}$ which divide $\mathfrak{c}$.  They are finite in number.


\begin{thm}(Cohen \cite{Cohen50})
A nearly Dedekind domain has finite rank.
\end{thm}
\noindent
Let $(R,\mm)$ be a one-dimensional local Noetherian domain.  Then the sequence $\{\dim_{R/\mm} \mm^i/\mm^{i+1}\}_{i=1}^{\infty}$ 
is eventually constant \cite[p. 40]{Sally78}; its eventual value is the \textbf{multiplicity}  $e(R)$ of $R$.  We put $e_{\pp}(R) = e(R_{\pp})$ and $e(R) = \sup_{\pp \in \MaxSpec R} e_{\pp}(R)$.

\begin{example}(Sally \cite[p. 5]{Sally78}) For a field $k$, put $R = k[[t^3,t^4]]$.  Then $R$ is a one-dimensional Noetherian 
local domain with maximal ideal $\pp = \langle t^3,t^4 \rangle$ and $R/\pp = k$.  Moreover we have $\pp^i = \langle t^{3i}, t^{3i+1}, t^{3i+2} \rangle = t^{3i} k[[t]]$ for 
all $i \geq 2$, so 
\[\mu(\pp) =  \dim_k \pp/\pp^2 = 2; \ \forall i \geq 2, \ \mu(\pp^i) = \dim_k \pp^i/\pp^{i+1} = 3. \]
Thus $z(R) = 2 < 3 = e(R)$.
\end{example}
\noindent
The following result computes the rank function in terms of the local multiplicities.

\begin{thm}
\label{FIRSTBIGTHM}
Let $R$ be a one-dimensional, non-normal Noetherian domain.  Then
\[  \rk(R) = \sup_{\pp \in \MaxSpec R} \rk(R_{\pp}) = e(R) . \]
\end{thm}
\begin{proof}
For all $\pp \in \MaxSpec R$ we have $e_{\pp}(R) \leq \rk(R_{\pp})$ and thus \[e(R) \leq  \sup_{\pp \in \MaxSpec R} \rk(R_{\pp}). \]
By Remark \ref{REMARK1} we have \[\sup_{\pp \in \MaxSpec R} \rk(R_{\pp}) \leq \rk(R). \]  Let $I$ be a nonprincipal ideal of $R$.  The Forster-Swan Theorem gives 
\[ \mu(I) \leq \sup_{\pp \in \MaxSpec R} \mu(I R_{\pp}). \]
The domain $R_{\pp}$ is one-dimensional local Cohen-Macaulay, 
so by \cite[Thm. 3.1.1]{Sally78} we have $\mu(I R_{\pp}) \leq e_{\pp}(R)$.  Thus $\mu(I) \leq e(R)$ and so $\rk(R) \leq e(R)$.
\end{proof}

\begin{remark}
In Theorem \ref{FIRSTBIGTHM} one can have $\rk(R) = \aleph_0$ \cite[pp. 38-40]{Cohen50}. 
\end{remark}

\section{Nonmaximal Orders}

\subsection{A First Example}
\textbf{} \\ \\ \noindent
One knows examples of local nearly Dedekind domains $R$ with multiplicity $e(R)$ any given
 $n \in \Z^+$: e.g. \cite{Watanabe73}.  The following is perhaps the most familiar.

\begin{example}
\label{EXAMPLE4.1}
For a field $k$ and $n \in \Z^+$, let \[R_n = k[[t^n,t^{n+1},\ldots,t^{2n-1}]] = k[[t^n]] + t^n k[[t]] = k+t^nk[[t]]. \]  Then $R_n$ is local nearly Dedekind with maximal ideal $\mm = \langle t^n,\ldots,t^{2n-1} \rangle = t^n k[[t]]]$ and $R/\mm = k$.  For $i \in \Z^+$ 
we have $\mm^i = \langle t^{in},\ldots,t^{(i+1)n-1} \rangle = t^{in} k[[t]]$, so 
\[ \rk(R) = e(R) = \lim_{i \ra \infty} \dim_{R/\mm} \mm^i/\mm^{i+1} = \lim_{i \ra \infty} n = n  = z(R). \]
\end{example}


\subsection{Nonmaximal Orders I: Maximal Rank}
\textbf{} \\ \\ \noindent
Let $A$ be a PID with fraction field $F$, and let $K/F$ be a field extension of degree $N \in \Z^{\geq 2}$.  An \textbf{A-order in K} 
is an $A$-subalgebra $R$ of $K$ which is finitely generated as an $A$-module and such that $F \otimes_A R = K$.  We say $R$ is an \textbf{A-order of degree N}.   The structure theory of modules over a PID implies $R \cong_A A^N$.  \\ \indent
Let $R$ be an $A$-order in $K$.  Then its normalization $\overline{R}$ is the integral closure of $A$ in $K$.  By Krull-Akizuki, $\overline{R}$ is a Dedekind domain.  If $\overline{R}$ is finitely generated as an $A$-module then it is an $R$-order in $K$, the unique maximal order.  It can happen that $\overline{R}$ is not a finitely generated $A$-module, but $\overline{R}$ is finitely generated if $K/F$ is separable 
or $A$ is finitely generated over a field.  

\begin{remark}
\label{REMARK4.2}
Let $A$ be a PID, not a field, with fraction field $F$, and let $K/F$ be a field extension of degree $N \in \Z^{\geq 2}$.  If the integral closure $S$ of $A$ in $K$ is not finitely generated as an $A$-module, then $K$ admits no normal $A$-order.  But it always admits some $A$-order: 
start with an $F$-basis of $K$, scale to get an $F$-basis $\alpha_1,\ldots,\alpha_N$ of elements integral over $A$, and take $S = A[\alpha_1,\ldots,\alpha_N]$.
\end{remark}

\begin{prop}
\label{4.2}
Let $N \in \Z^{\geq 2}$, let $A$ be a PID, and let $R$ be a non-normal $A$-order of degree $N$.   If there is $\pp \in \MaxSpec R$ 
such that $z_{\pp}(R) = N$, then $\rk(R) = N$.
\end{prop}
\begin{proof}
Because $R$ is free of rank $N$ as a module over the PID $A$, every ideal of $I$ of $R$ is a free $R$-module of rank at most $N$ 
and thus $\mu(I) \leq N$. so $\rk(R) \leq N$.  On the other hand $\rk(R) \geq e_{\pp}(R) \geq z_{\pp}(R) = N$.
\end{proof}
\noindent
We say an $A$-order in $K$ has \textbf{maximal rank} 
if $\rk(R) = N = [K:F]$.  

\begin{thm}
\label{SECONDBIGTHM}
Let $A$ be a PID with fraction field $F \supsetneq A$, and let $K/F$ be a field extension of degree $N \in \Z^{\geq 2}$.  Then there is an 
$A$-order $R$ in $K$ of maximal rank.
\end{thm}
\begin{proof}
By Remark \ref{REMARK4.2}, there is an $A$-order $S$ in $K$.  Let $x$ be a nonzero, nonunit in $A$, so $x = p_1^{a_1} \cdots p_r^{a_r}$ where $p_1,\ldots,p_r$ are nonassociate 
prime elements.  Put 
\[ R = R(S,x) = A + x S. \]
We {\sc claim}: (i) $R$ is an $A$-order; (ii) for $1 \leq i \leq r$ there is a unique $\pp_i \in \Spec R$
with $\pp_i \cap A = (p_i)$; and (iii) $z_{\pp_i}(R) = N$ for all $i$.  Assuming the claim, Proposition \ref{4.2} gives $\rk(R) = N$.  We show the claim:
\\ \indent
(i) Certainly $R$ is a subring of $K$.  If $\alpha_1 = 1$, $\alpha_2,\ldots,\alpha_N$ is an $A$-basis for $S$,\footnote{We are permitted to take $\alpha_1 = 1$ by 
``Hermite's Lemma'' \cite[Prop. 6.14]{CA}.} then $1,x\alpha_2,\ldots,x\alpha_N$ is an $A$-basis for $R_x$.  So $R$ is an $A$-order in 
$K$.  \\ \indent
(ii) Fix $1 \leq i \leq r$ and define $\pp_i = p_iA + x S$.  Then $\pp_i$ is an ideal of $R$ and $R/\pp_i \cong A/(p_i)$ is a field, 
so $\pp_i$ is a prime ideal of $R$ containing $p$.  Let $\qq$ be a prime ideal of $R$ such that $\qq \cap A = (p_i)$.  Then since $p_i \mid x$ we have
\[ \pp_i^2 = (p_i A + xS)^2 = p_i^2 A + p_i x S + x^2 S = p_i^2 A + p_i x S \subset p_i R \subset \qq. \]
Since $\qq$ is prime we get $\qq \supset \pp_i$ and thus (since $\dim R = 1$) $\qq = \pp_i$. \\ \indent
(iii) Since $p_i,x\alpha_2,\ldots,x\alpha_N$ is an $A$-basis for $\pp_i$ and $p_i^2,p_ix_2,\ldots,p_i x_N$ is an $A$-basis for $\pp_i^2$,
we have $ \pp_i/\pp_i^2 \cong_A (A/(p_i))^N$.  Thus for all $1 \leq i \leq r$ we have
\[ z_{\pp_i}(R) = \dim_{R/\pp_i} \pp_i/\pp_i^2 = \dim_{A/(p_i)} (A/(p_i))^N = N.  \qedhere\]
\end{proof}

\begin{remark}
a) If $R$ is a PID with fraction field $F \supsetneq R$, then $F$ admits a degree $N$ field extension for all $N \in \Z^{\geq 2}$:  for $(p) \in \MaxSpec R$, we can take $K = F(p^{\frac{1}{N}})$. 
b) The $R = A + xS$ construction is modelled on \cite[Thm. 3.15]{Conrad}.
\end{remark}

\begin{example}
\label{EXAMPLE3.6}
Let $k$ be a field.  For $n \in \Z^{\geq 2}$ put $A = k[[t^n]]$, so $F = k((t^n))$, and let $K = k((t))$.  Let $S =k[[t]]$, the maximal 
$A$-order in $K$.  Then $R(S,t^n) = k[[t^n] + t^n k[[t]]$ is the ring $R_n$ of Example \ref{EXAMPLE4.1}.
\end{example}

\begin{example}
\label{EXAMPLE3.7}
For $n \in \Z^{\geq 2}$ put $A = \Z$, $K = \Q(2^{\frac{1}{n}})$, 
$S = \Z[2^{\frac{1}{n}}]$ and $x = 2$.  Then $R = R(S,x) = \Z[2^{\frac{n+1}{n}}, 2^{\frac{n+2}{n}},\ldots,2^{\frac{2n-1}{n}}]$ 
has rank $n$, and it is the order in $K$ that Matson used to prove Theorem \ref{MATSON}b).   
\end{example}

\subsection{Nonmaximal Orders II: Pullbacks and Locally Maximal Orders}
\textbf{} \\ \\ \noindent
Let $A$ be a DVR with fraction field $F$ and residue field $R/\mm = k$.  Let $l/k$ be a separable field extension of 
degree $d \in \Z^{\geq 2}$, let $K/F$ be the corresponding degree $n$ unramified (hence separable) field extension, and let $S$ be the integral closure of $A$ in $K$, so $S$ is a DVR with maximal ideal $\overline{\mm}$ (say) and $S/\overline{\mm} = l$.  Let $q: S \ra l$ be 
the quotient map, and put $R = q^{-1}(k)$.  By \cite[pp. 35-36]{Anderson-Mott92} and the Eakin-Nagata Theorem \cite[Cor. 8.31]{CA}, $R$ is local nearly Dedekind with normalization $S$ and an $A$-order in $K$ of rank $d$.  In \emph{loc. cit.}, Anderson and Mott show $\mm = q^{-1}(0) = \overline{\mm}$ and $R/\mm = k$.  Finally,  
$\mm/\mm^2 = \overline{\mm}/\overline{\mm}^2$ is a one-dimensional $l$-vector space hence a $d$-dimensional $k = R/\mm$-vector space, so Theorem \ref{SECONDBIGTHM} applies: we have $\rk(R) = d$.

\begin{thm}
\label{THIRDBIGTHM}
For $r \in \Z^+$, prime numbers $p_1,\ldots,p_r$ and $n_1,\ldots,n_r \in \Z^{\geq 2}$ with 
$\max_i n_i = n$, there is a degree $N$ number field $K$ and a $\Z$-order $R$ in $K$ with exactly $r$ singular primes $\pp_1,\ldots,\pp_r$, 
such that $e_{\pp_i}(R) = n_i$ for all $i$.  Thus $\rk(R) = n$.  
\end{thm}
\begin{proof} 
Step 1: The construction of $K$ and $R$ is essentially given in \cite[$\S$3.2-3.3]{CGP15}; 
the only difference is that in that construction one inverts all the regular primes to get a semilocal domain.   So we will content ourselves 
with a sketch.  For $1 \leq i \leq r$, let $\Q_{p_i^{n_i}}$ be the unramified extension of $\Q_{p_i}$ of degree $n_i$.  By weak approximation / Krasner's Lemma there is a number field $K$ of degree $N$ and primes $\mathcal{P}_1,\ldots,\mathcal{P}_r$ such that 
$K_{\mathcal{P}_i} \cong \Q_{p_i^{n_i}}$ for all $i$.  The local degree $[K_{\mathcal{P}_i}:\Q_{p_i}] = n_i$ is assumed (only) to be at most $N$; when $n_i < N$ this is handled by having other primes of $\Z_K$ lying over $p_i$.  We have 
\[ \Z_K/\mathcal{P}_1 \cdots \mathcal{P}_r \cong \prod_{i=1}^r \Z_K/\mathcal{P}_i \cong \prod_{i=1}^r \F_{p_i^{n_i}}. \]
Let $q: \Z_K \ra \prod_{i=1}^r \F_{p_i^{n_i}}$ be the corresponding quotient map.   Then we take 
\[ R = q^{-1}(\prod_{i=1}^r \F_{p_i}). \]
For $1 \leq i \leq r$, let $\pp_i = \mathcal{P}_i \cap R$ and let $q_i: (\Z_K)_{\mathcal{P}_i} \ra (\Z_K)_{\mathcal{P}_i}/\mathcal{P}_i
\cong \F_{p_i^{n_i}}$.  Then (as is shown in detail in \emph{loc. cit.}) for all $i$, the ring $R_{\pp_i}$ is the pullback $q_i^{-1}(\F_{p_i})$, and thus is local nearly Dedekind 
with $e_{\pp_i}(R) = e(R_{\pp_i}) = [\F_{p_i^{n_i}}:\F_{p_i}] = n_i$. \\
Step 2:  We have -- e.g. using CRT as above -- that $\pp_1,\ldots,\pp_r$ 
are the only singular primes of $R$, so 
\[z(R) = \max_{1 \leq i \leq r} n_i = n. \]
For  $1 \leq i \leq r$, let $\widehat{R_{\pp_i}}$ denote the $\pp_i$-adic completion of 
$R_{\pp_i}$.  Because there is a unique prime of $\overline{R}$ lying over $\pp_i$ (``analytically irreducible''), $\widehat{R_{\pp_i}}$ is 
itself a one-dimensional local Noetherian domain, with fraction field $K_{\pp_i}$, and moreover we have 
$e(R_{\pp_i}) = e(\widehat{R_{\pp_i}})$ and  $z(\widehat{R_{\pp_i}})  = z(R_{\pp_i}) = n_i$.   But $\widehat{R_{\pp_i}}$ is a $\Z_{p_i}$-order 
in $K_{\pp_i}$ of degree $n_i$, so Theorem \ref{SECONDBIGTHM} applies to give 
\[ \rk(R_{\pp_i}) = e(R_{\pp_i}) = e(\widehat{R_{\pp_i}}) = \rk(\widehat{R_{\pp_i}}) = z(\widehat{R_{\pp_i}}) = z(R_{\pp_i}) = z_{\pp}(R) \]
and thus 
\[ \rk(R) = \max_{1 \leq i \leq r} \rk(R_{\pp_i}) = \max_{1 \leq i \leq r}  z_{\pp_i}(R) = \max_{1 \leq i \leq r} n_i = n.  \qedhere\]
\end{proof}

\begin{remark}
a) There is an analogue of Theorem \ref{THIRDBIGTHM} with $\Z$ replaced by $\F_q[t]$.  \\
b) One would like to take $n = 1$ in Theorem \ref{THIRDBIGTHM}!   Unfortunately at present 
we cannot prove that there are infinitely many number fields with class number one: this is perhaps the most (in)famous 
open problem in algebraic number theory.
\end{remark}




\section{More Rings of Finite Rank}

\subsection{A Trichotomy}

\begin{thm}
\label{4.1}
a) Let $R = R_1 \times \cdots \times R_n$ be a finite direct product of rings.  Then 
\[ \rk(R) = \max_i \rk(R_i). \]
b) (Gilmer) For a Noetherian ring $R$, the following are equivalent: \\
(i) $R$ has finite rank. \\
(ii) For all minimal primes $\pp \in \Spec R$, the ring $R/\pp$ has finite rank.  \\
c) (Cohen-Gilmer) A ring of finite rank has dimension zero or one.
\end{thm}
\begin{proof}
a) Every ideal in $R_1 \times \cdots \times R_n$ is of the form $I_1 \times \cdots \times I_n$ for ideals $I_i$ of $R_i$.   b) This is \cite[Thm. 2]{Gilmer72}.   c) Apply Theorem \ref{FIRSTCOHEN} and part b).
\end{proof}
\noindent
So if $R$ is a ring with $\rk(R) \in \Z^+$, exactly one of the following holds: \\
$\bullet$ $R$ is Noetherian of dimension zero, i.e., Artinian; \\
$\bullet$ $R$ is a Noetherian domain of dimension one; \\
$\bullet$ $R$ is Noetherian of dimension one and not a domain.  
\\ \\
We treated  domains in $\S$3, and we will discuss Artinian rings in $\S$4.2.  This leaves us with rings 
which are one-dimensional Noetherian and not domains.  One can show that there are such rings of all ranks quite cheaply: 
if $R$ is a domain of rank $n \in \Z^+$ then Theorem \ref{4.1}a) gives $\rk(R \times R) = n$.  The more 
interesting case is when the localization of $R$ at some minimal prime is not a domain, so in particular when $R$ has a unique 
minimal prime, which is nonzero.  A good example is a polynomial ring over a local Artinian ring.  In $\S$4.3 we study the ranks of polynomial rings.

\subsection{Artinian Rings}
\textbf{} \\ \\ \noindent
Let $\rr$ be an Artinian ring.  We denote by $\ell(\rr)$ the length of $\rr$ as an $\rr$-module, 
which is finite.  We denote by $n(\rr)$ the \textbf{nilpotency index} of $\rr$: the least $n \in \Z^+$ such that if $x \in \rr$ is nilpotent 
then $x^n = 0$.

\begin{prop}
\label{4.2}
Let $\rr$ be an Artinian ring.  \\
a) We have $\rk(\rr) \leq \ell(R)$.  \\
b) If $\ell(\rr) \geq 2$ (i.e., if $\rr$ is nonzero and not a field), then $\rk(\rr) \leq \ell(\rr)-1$.
\end{prop}
\begin{proof}
The result is clear when $\rr$ is the zero ring or is a field, so assume $\ell(\rr) \geq 2$.  If $\rr$ is principal then $\rk(\rr) = 1 \leq \ell(\rr)-1$.
 If $\rr$ is not 
principal, then there is an ideal $I$ with $2 \leq \mu(I) \leq \ell(I)$, and such an ideal is necessarily proper, so $\ell(\rr) \geq \ell(I)+1$.
\end{proof}
\noindent
The following result shows that the bound of Proposition \ref{4.2}b) is sharp.

\begin{cor}
\label{4.3}
Let $n \in \Z^{\geq 2}$.  Then: \\
a) For any field $k$, there is an Artinian $k$-algebra of rank $n$ and length $n+1$.  \\
b) There is a finite ring $R$ of rank $n$ and length $n+1$.  
\end{cor}
\begin{proof}
Let $R$ be a non-normal domain of finite rank $n$ with a maximal ideal $\pp$ such that $\dim_{R/\pp} \pp/\pp^2 = n$.  Then 
$R/\pp^2$ is Artinian, of rank $n$ and length $n+1$.   Taking $R$ as in Example \ref{EXAMPLE3.6} 
gives part a), and  taking $R$ as in Example \ref{EXAMPLE3.7} gives part b).  
\end{proof}

\subsection{Polynomial Rings of Finite Rank}
\textbf{} \\ \\ \noindent
Next we explore polynomial rings of finite rank.  For any nonzero ring $\mathfrak{r}$, a polynomial ring over $\mathfrak{r}$ 
in at least two indeterminates has Krull dimension at least $2$ and thus has infinite rank.  So we are 
reduced to the case of $\mathfrak{r}[t]$. 

\begin{thm}
\label{4.4}
For a ring $\rr$, the following are equivalent: \\
(i) The polynomial ring $\rr[t]$ has finite rank. \\
(ii) The ring $\rr$ is Artinian. 
\end{thm}
\begin{proof}
(i) $\implies$ (ii): Suppose $\rr[t]$ has finite rank.  Then $\rr[t]$ is Noetherian, hence so is $\rr$.  Thus 
$\dim \rr[t] = 1 + \dim \rr$.  By Theorem \ref{FIRSTCOHEN} we have $\dim \rr = 0$, os $\rr$ is Artinian. \\
(ii) $\implies$ (i): If $\rr$ is Artinian, there are local Artinian rings $\rr_1,\ldots,\rr_r$ such that $\rr \cong \rr_1 \times \cdots 
\times \rr_r$, and then $\rr[t] \cong \rr_1[t] \times \cdots \times \rr_r[t]$.  So Theorem \ref{4.1}a) reduces us to showing: 
a polynomial ring $\rr[t]$ over a local Artinian ring $(\rr,\mm)$ has finite rank.  The ring $\rr[t]$ is Noetherian, with unique minimal prime  $\mm \rr[t]$.  Since $\rr[t]/\mm\rr[t] = (R/\mm)[t]$ is a PID, the ring $\rr[t]$ 
has finite rank by Theorem \ref{4.1}b).
\end{proof}

\begin{thm}
\label{4.5}
Let $\rr$ be an Artinian local ring of length $\ell$, and let $R = \rr[t]$.  Then 
\[ \rk R \leq \ell. \]
\end{thm}
\begin{proof}
Let $\pp$ be the unique prime ideal of $\rr$,  let $k = \rr/\pp$.  Then $\mathcal{P} = \pp[t]$ is the unique minimal prime of $R$.   By Theorem \ref{4.4}, $R$ has finite rank.    The main input  was Theorem \ref{4.1}b), and we will get the upper bound $\rk R \leq \ell$ by following Gilmer's proof.  
For $1 \leq i \leq \ell$, let 
\[ R_i = R/\mathcal{P}^i = \rr[t]/\pp^i \rr[t] = \rr/\pp^i [t]. \]
  We will show inductively that $\rk R_i$  is at most the 
length $\ell(\rr/\pp^i)$ of $\rr/\pp^i$.  \\
Base Case $(i = 1)$: The ring $R_1 = \rr/\pp [t] = k[t]$ is a PID, thus $\rk R_1 = 1 = \ell(k)$.  \\
Inductive Step: Suppose $1 \leq i < \ell$ and $\rk R_i \leq \ell(\rr/\pp^i)$.  Consider the exact sequence of $R$-modules 
\[ 0 \ra \mathcal{P}^i/\mathcal{P}^{i+1} \ra R/\mathcal{P}^{i+1} \ra R/\mathcal{P}^i \ra 0. \]
For any surjective homomorphism of rings $R \ra S$, the rank $\mu_*(S)$ as an $R$-module is equal to its rank $\rk S$ 
as a ring.    Thus our inductive hypothesis gives $\mu_*(R/\mathcal{P}^i) \leq \ell(\rr/\pp^i)$.   Moreover the rank of $\mathcal{P}^i/\mathcal{P}^{i+1}$ as an $R$-module is its rank as an $R/\mathcal{P} = k[t]$-module.  By \cite[Prop. 2]{Gilmer72} we have 
\[ \mu_*(\mathcal{P}^i/\mathcal{P}^{i+1}) \leq \rk(k[t]) \mu_R(\mathcal{P}^i/\mathcal{P}^{i+1}) = \mu_{\rr}(\pp^i/\pp^{i+1}) = 
\ell(\pp^i/\pp^{i+1}). \]
By \cite[Prop. 1]{Gilmer72} we have 
\[ \rk R_{i+1} = \mu_*(R/\mathcal{P}^{i+1}) \leq \mu_*(\mathcal{P}^i/\mathcal{P}^{i+1}) + \mu_*(R/\mathcal{P}^i) 
\leq \ell(\pp^i/\pp^{i+1}) + \ell(\rr/\pp^i) = \ell(\rr/\pp^{i+1}), \]
completing the induction.  Since the nilpotency index of $\rr$ is at most $\ell$, we have
\[ \rk R =  \rk \rr[t]  = \rk \rr/\pp^{\ell} [t] = \rk R_{\ell} \leq \ell(\rr/\pp^{\ell}) = \ell(\rr) = \ell. \qedhere \]
\end{proof}
\noindent
Since an Artinian ring is a finite product of local Artinian rings, Theorem \ref{4.5} implies: for any Artinian ring $\rr$, the 
rank of $\rr[t]$ is bounded above by the maximum length of the local Artinian factors of $\rr$.

\begin{cor}
Let $\rr$ be a nonzero principal Artinian ring, which we may decompose as $\rr = \prod_{i=1}^r \rr_i$, with each $\rr_i$ a local 
Artinian principal ring.   For $1 \leq i \leq r$, let $n_i$ be the length of $\rr_i$ (which coincides with its nilpotency index), and let 
$n = \max_{1 \leq i \leq r} n_i$.  Then $\rk \rr[t] = n$.
\end{cor}
\begin{proof}
We easily reduce to the case in which $\rr$ 
is local with maximal ideal $\pp = (\pi)$.  Theorem \ref{4.5} gives $\rk \rr[t] \leq n$.  Let $\mm = \langle \pi, t \rangle$, so $\mm$ is a maximal ideal of $\rr[t]$ and 
$R/\mm = \rr/\pi \rr = k$, say.  We claim that the ideal $\mm^{n-1} = \langle \pi^{n-1}, \pi^{n-2} t,
\ldots, \pi t^{n-2}, t^{n-1} \rangle$ requires $n$ generators.  Indeed, observe that for $a_1,\ldots,a_{n} \in \rr^{\times}$, 
$a_1 \pi^{n-1} + a_2 \pi^{n-2} t + \ldots + a_{n} t^{n-1}$ cannot be expressed as an $\rr$-linear combination of terms $\pi^i t^{j}$ with $i+j \geq n$, so $\dim_k \mm^{n-1}/\mm^{n} = n$ and $\mu(\mm^{n-1}) = n$.
\end{proof}

\begin{remark}
\label{4.6}
a) Theorem \ref{4.5} implies: if $A$ is a PID, $\pi \in R$ a prime element and $n \in \Z^+$, then $\rk A/\pi^n A[t] = n$.  In fact this 
is equivalent: every local principal Artinian ring is a quotient of a PID \cite[Cor. 11]{Hungerford68}.   P. Pollack has shown me a thoroughly elementary proof that $\rk A/\pi^n A[t] \leq n$.  \\
b) In particular, for a prime number $p$ and $n \in \Z^+$ we have 
\[ \rk \Z/p^n \Z[t] = n. \]
The case $p = n = 2$ appears in Matson's thesis \cite[p. 44, Example 1.3.15]{Matson}.  Her proof uses that $\Z/4\Z$ has a unique nonzero zero-divisor.  The only other ring with this property is $\Z/2\Z[t]/(t^2)$, so this argument is rather specialized.  Matson's result was our motivation for finding Theorem \ref{4.5}. 
\end{remark}

\end{document}